\documentclass[amstex,12pt,reqno]{amsart}
\usepackage{amsmath,amsfonts,amssymb,amsthm,hyperref,enumerate,multicol}
\usepackage{graphicx}
\newtheorem{lemma}{Lemma}
\newtheorem{theorem}{Theorem}

\numberwithin{equation}{section}

\begin{document}

\dedicatory{In memory of a beloved man and an Approx. Theorist Prof. D. V. Pai}

\leftline{ \scriptsize \it}
\title[]
{Some Fuzzy Korovkin type Approximation Theorems via Power Series Summability Method}
\maketitle
\begin{center}
{ \bf B. Baxhaku$^1$, P. N. Agrawal$^2$ and R. Shukla$^3$ \footnote{Corresponding author - Rahul Shukla}} \\
\vskip0.15in
$^{1}$Department of Mathematics, University of Prishtina, Kosovo\\
\vskip.15in
$^{2}$Department of Mathematics, Indian Institute of Technology Roorkee, India\\
\vskip.15in
$^{3}$Department of Mathematics, University of Delhi, India\\ 
\vskip.15in
emails: $^1$ behar.baxhaku@uni-pr.edu, $^2$p.agrawal@ma.iitr.ac.in, and $^3$rshukla@ma.iitr.ac.in
\end{center}


\begin{abstract}
This article provides a power series summability based Korovkin type approximation theorem for any fuzzy sequence of positive linear operators. Using the notion of fuzzy modulus of smoothness, we also derive an associated approximation theorem concerning the fuzzy rate of convergence of these operators. Furthermore, through an example, we illustrate that our summability based Korovin type theorem is an advantage over the fuzzy Korovkin type theorem proved in the seminal paper by Anastassiou (Stud. Univ. Babe\c{s}-Bolyai Math. 4 (2005), 3--10).
\end{abstract}


\section{Introduction}
The concept of fuzziness is a famous key to measure the completeness of a certain assumption or concern. Notably, its applications occur everywhere around us in the areas like stock market, bio-medicines, weather predictions, robotics, and pattern recognitions, etc. It has been observed by the approximation theorists that one of the serious issues in classical approximation theory is to address the limit (or the statistical limit) with the absolute accuracy. To deal with this, fuzzy analogues of many classical approximation theorems have been established (please see \cite{Anas2}), and these efforts have laid the formal foundation of the Fuzzy Approximation Theory.


Zadeh \cite{Zadeh} was the first mathematician who systematically introduced and developed the elements of fuzzy set theory, particularly membership function and fuzzy numbers. Goetschel and Voxman \cite{GV} presented a slightly modified definition of fuzzy numbers and also defined a metric for this family of fuzzy sets. Subsequently, researchers defined different classes of sequences of fuzzy numbers and studied them from different perspectives, for instance (cf. \cite{DP, EA1, EA2, Savas, TB1, TB2}). Inspired by the work of Gal \cite{Gal}, Anastassiou \cite{Anas1} proved a basic fuzzy Korovkin type theorem using Shisha-Mond inequality and established a fuzzy rate of convergence in terms of fuzzy modulus of continuity. Anastassiou \cite{Anas3} studied some fuzzy-wavelet type operators and fuzzy-neural network type operators. Furthermore, the author established a higher order point-wise approximation result in fuzzy sense. Using the A-statistical convergence technique, Anastassiou and Duman \cite{Anas4} proposed some fuzzy Korovkin type results and showed that their results have an advantage over the corresponding results in classical approximation. Recently, Yavuz \cite{Yavuz} established a fuzzy trigonometric Korovkin type approximation theorem by using power series summability method, and also derived another approximation result for fuzzy periodic continuous functions with the help of fuzzy modulus of continuity. Using an example concerning fuzzy Abel-Poisson convolution operators, the author showed that the result established in \cite{ Anas} (Theorem 1.1) does not work but his result obtained by means of power series method works. In other words, the author's result is a non-trivial generalization of the result given in \cite{Anas}. For a systematic exposition of the idea of fuzzy approximation theory, we highly recommend the readers to follow the classic book by Anastassiou \cite{Anas2}.

Inspired by these studies, we propose to investigate the fuzzy Korovkin type approximation theorems for fuzzy continuous functions on a compact support $J=[a,b]\subset\mathbb{R}$ by means of power series method. Before proceeding further, we recall some basic definitions of fuzzy set theory and the power series summability method in the fuzzy number space. 

\section{Important Keywords and Power Series Summability}
In fuzzy mathematics, a membership function is a map $f: \mathbb{R} \to [0,1]$ which decides the grade of membership of any real number in the given fuzzy set $A$. Subsequently, any membership function is said to be a fuzzy number if it is normal, fuzzy convex, upper semi-continuous, and the closure of the set $\{t\in\mathbb{R}: f(t)>0\}$ is compact. We denote the set of all fuzzy real numbers by the analogous notation $\mathbb{R}_{\mathcal{F}}$. For any $x\in \mathbb{R}_{\mathcal{F}}$, we have the following $\alpha$-level fuzzy sets:     
\begin{eqnarray*}
[x]_{_\alpha} = \begin{cases}
                   \{t\in \mathbb{R}:x(t)\geq \alpha\}, & \hbox{for $0<\alpha\leq 1$;} \\\\
                   \overline{\{t\in \mathbb{R}:x(t)> \alpha\}}, & \hbox{for $\alpha=0$.}
                 \end{cases}               
\end{eqnarray*}
It is well known \cite{RW} that the set $[x]_{_\alpha}$ is compact for each $\alpha$. Furthermore, for all $x, y\in \mathbb{R}_{\mathcal{F}}$ and $\lambda\in \mathbb{R}$, the algebra on the $\alpha$-level sets is defined as follows:
$$[x\bigoplus y]_{\alpha}=[x]_{\alpha}+[y]_{\alpha}, [\lambda\bigodot x]_{\alpha}=\lambda[x]_{\alpha},$$
where $[z]_{\alpha}=[z_{\alpha}^{-},z_{\alpha}^{+}]$, and $z_{\alpha}^{-}, z_{\alpha}^{+}$ are the end points of $[z]_{\alpha}$ for each $\alpha\in [0,1]$. In addition, we have a partial ordering in $\mathbb{R}_{\mathcal{F}}$ by $x\preceq y$, if and only if $x_{\alpha}^{-}\preceq y_{\alpha}^{-}$ and $x_{\alpha}^{+}\preceq y_{\alpha}^{+},~\forall~\alpha\in [0,1].$ We have a metric $D:\mathbb{R}_{\mathcal{F}}\times\mathbb{R}_{\mathcal{F}}\rightarrow\mathbb{R}^{+}\cup \{0\}$ as
$$D(x,y)=\sup_{\alpha\in [0,1]}d([x]_{\alpha},[y]_{\alpha}),$$
where $d$ is the usual Hausdorff metric given by
$$d([x]_{\alpha},[y]_{\alpha})=max\{|x_{\alpha}^{-}-y_{\alpha}^{-}|, |x_{\alpha}^{+}-y_{\alpha}^{+}|\}.$$
From \cite{Congxin}, it is known that $(\mathbb{R}_{\mathcal{F}},D)$ is a complete metric space.\\ For the fuzzy number valued functions $g,h:J\subset \mathbb{R}\rightarrow \mathbb{R}_{\mathcal{F}}$, the distance between $g$ and $h$ is defined by
$$D^{*}(g,h)=\sup_{t\in J}D(g(t),h(t)).$$
Let $C^{\mathcal{F}}(J)$ be the space of all fuzzy continuous functions on the interval $J$, then the operator $L:C^{\mathcal{F}}(J)\rightarrow C^{\mathcal{F}}(J)$ is called a fuzzy linear operator provided
$$L(a_1\bigodot g_1\bigoplus a_2\bigodot g_2)=a_1\bigodot L(g_1)\bigoplus a_2\bigodot L(g_2),$$
where $a_1,a_2\in \mathbb{R}$ and $g_1,g_2\in C^{\mathcal{F}}(J)$. Further, the fuzzy linear operator $L$ is said to be fuzzy positive linear operator if for any $g,h\in C_{\mathcal{F}}(J)$, with $g(t)\preceq h(t)$, $\forall ~ t\in J$, we have $L(g;t)\preceq L(h;t).$
\noindent 
The power series summability method (or the P-Summability) was introduced by Borwein \cite{Bor}, which includes the Abel and Borel summability methods as special cases. Many researchers have used the P-summability to establish Korovkin type approximation theorems in the settings of various function spaces, some recent works can be found in \cite{PN, PN1, KDe, Tas1, Tas}. Recently, Sezer and $\c{C}$anak \cite{SC} have applied the P-summability to the space of fuzzy numbers $\mathbb{R}_{\mathcal{F}}$ and given some Tauberian theorems based on the said summability method. We recall their nice P-summability based convergence definition as follows:

Consider the power series $p(t)=\sum\limits_{i=1}^{\infty}p_i t^{i-1}$ corresponding to the non-negative sequence $\big{<}p_i;p_1>0\big{>}$ with the radius of convergence 1, provided $\sum_{i=1}^{n}p_i\rightarrow \infty$, as $n\rightarrow \infty$. A sequence $(u_n)$ of fuzzy numbers is said to be P-summable to some fuzzy number $\mu$ determined by $p$ if $\sum\limits_{n=1}^{\infty}u_np_n t^{n-1}$ converges for $t\in(0,1)$ and
$$\lim\limits_{t\to 1^-}\frac{1}{p(t)}\sum\limits_{n=1}^{\infty}u_np_nt^{n-1}=\mu.$$  
\noindent
In our further consideration, let $C(J)$ denote the space of all continuous functions on $J$ with the uniform norm $\|.\|.$
The prime aim of the present paper is to examine a fuzzy Korovkin type approximation theorem in the algebraic case by using the fuzzy P-summability method and also prove a related approximation theorem for fuzzy continuous functions on a closed and bounded interval by means of fuzzy modulus of continuity. We also give an example to show that our Korovkin type theorem is a non-trivial generalization of the basic fuzzy Korovkin theorem.

\section{Fuzzy Korovkin Theorem via Power Series Method}
\noindent
For the better understanding of our main results, we shall recall the following standard Korovkin theorem for fuzzy positive linear operators:
%
\begin{theorem}(\cite{Anas1}, Thm. 4)\label{thm.asym.voronovaskaya}
Let $\{\mathcal{T}_{n}\}_{n\in \mathbb{N}}$ be a sequence of fuzzy positive linear operators from $C^{(\mathcal{F})}(J)$ into itself. Assume that there exists a corresponding sequence $\{\overline{\rm \mathcal{T}}_{n}\}_{n\in \mathbb{N}}$ of positive linear operators from $C(J)$ into itself with the property
\begin{eqnarray}\label{vetia1}
\big\{\mathcal{T}_{n}(f;x)\big\}_{\alpha}^{\pm}=\overline{\rm \mathcal{T}}_{n}(f_{\alpha}^{\pm};x)
\end{eqnarray}	
respectively, for all $\alpha\in [0, 1],$ $\forall f\in C^{(\mathcal{F})}(J)$. Further, suppose that
\begin{eqnarray*}
	\lim\limits_{n\to \infty}||\overline{\rm \mathcal{T}}_{n}(e_i)-e_i||=0,
\end{eqnarray*}
for $i=0,1,2$ with $e_0(x)=1,$ $e_1(x)=x,$ $e_2(x)=x^2.$ Then, 
 for all $f\in C^{(\mathcal{F})}(J)$ we have
\begin{eqnarray*}
	\lim\limits_{n\to \infty}D^{*}\left( \mathcal{T}_{n}(f),f\right)=0.
\end{eqnarray*}

\end{theorem}
Let us consider the sequence $\{\mathcal{T}_n\}$ of fuzzy positive linear operators from  $C^{(\mathcal{F})}(J)$ into itself with the property (\ref{vetia1}). Further, let for each $t \in (0,1)$
\begin{eqnarray}\label{fz1}
\sum\limits_{n=1}^{\infty}\|\overline{\rm \mathcal{T}}_{n}(1)\|p_n t^{n-1}<\infty.
\end{eqnarray}
Then for all $f\in C^{(\mathcal{F})}(J)$, from \cite{SC} it is known that the series $\sum\limits_{n=1}^{\infty}\overline{\rm \mathcal{T}}_{n}(f_{\alpha}^{\pm})p_n t^{n-1}$ and series $\sum\limits_{n=1}^{\infty}{ \mathcal{T}}_{n}(f)p_n t^{n-1}$ converge for $t \in (0,1).$

We now present our first main result, the fuzzy Korovkin Theorem via power series method.

\begin{theorem}\label{thm.asym.voronovaskaya1}

Let $\{\mathcal{T}_{n}\}_{n\in \mathbb{N}}$ be a sequence of fuzzy positive linear operators from $C^{(\mathcal{F})}(J)$ into itself. Assume that there exists a corresponding sequence $\{\overline{\rm \mathcal{T}}_{n}\}_{n\in \mathbb{N}}$ of positive linear operators from $C(J)$ into itself with the properties (\ref{vetia1}) and (\ref{fz1}). Further, let
\begin{eqnarray}\label{eqthm0}
\lim\limits_{t\to 1^-}||\frac{1}{p(t)}\sum\limits_{n=1}^{\infty}p_{n}t^{n-1}\overline{\rm \mathcal{T}}_{n}(e_i)-e_i||=0,
\end{eqnarray}
for each test function $e_{i},\;i=0,1,2$. Then, for all $f\in C^{(\mathcal{F})}(J)$ we have
\begin{eqnarray}\label{eqthm1}
\lim\limits_{t\to 1^-}D^{*}\left(\frac{1}{p(t)}\sum\limits_{n=1}^{\infty}p_{n}t^{n-1} \mathcal{T}_{n}(f),f\right)=0.
\end{eqnarray}



\end{theorem}
\begin{proof} Let $f\in C^{(\mathcal{F})}(J)$, $\alpha\in [0,1$] and $x\in I$ be fixed. Suppose (\ref{eqthm0}) be satisfied and let $f\in C^{(\mathcal{F})}(J)$.  Then from the continuity of $f_{\alpha}^{\pm}\in C(I)$ it follows that for a given $\epsilon > 0$ there is a number $\delta > 0$ such that $\left|f_{\alpha}^{\pm}(z)-f_{\alpha}^{\pm}(x)\right|< \epsilon $ holds for every $z\in I$ satisfying $|z-x|<\delta.$ Then, for all $z\in I,$  we immediately get that
\begin{eqnarray}\label{eqthm3}
&&\left|f_{\alpha}^{\pm}(z)-f_{\alpha}^{\pm}(x)\right|\leq \epsilon +\frac{2 M_{\alpha}^{\pm}(z-x)^2}{\delta^2},
\end{eqnarray}
where $x\in J$ and $\alpha\in [0,1]$ where $M_{\alpha}^{\pm}=||f_{\alpha}^{\pm}||$ (see \cite{Korovkin}).  In view of the linearity and positivity of the operators $\overline{\rm \mathcal{T}}_{m,n}$, we have, for each $n\in \mathbb{N},$ that
\begin{eqnarray}\label{eqK2}
\left|\frac{1}{p(t)}\sum\limits_{n=1}^{\infty}\overline{\rm \mathcal{T}}_{n}(f_{\alpha}^{\pm};x)p_{n}t^{n-1}-f_{\alpha}^{\pm}(x)\right|&\leq& \frac{1}{p(t)}\sum\limits_{n=1}^{\infty}p_{n}t^{n-1}\overline{\rm \mathcal{T}}_{n}(|f_{\alpha}^{\pm}(z)-f_{\alpha}^{\pm}(x)|;x)\nonumber\\&+&M_{\alpha}^{\pm}\left|\frac{1}{p(t)}\sum\limits_{n=1}^{\infty}p_{n}t^{n-1}\overline{\rm \mathcal{T}}_{n}(1;x)-1\right|\nonumber\\&\leq&  \frac{1}{p(t)}\sum\limits_{n=1}^{\infty}p_{n}t^{n-1}\overline{\rm \mathcal{T}}_{n}\left(\epsilon +\frac{2 M_{\alpha}^{\pm}(z-x)^2}{\delta^2};x\right)\nonumber\\&+&M_{\alpha}^{\pm}\left|\frac{1}{p(t)}\sum\limits_{n=1}^{\infty}p_{n}t^{n-1}\overline{\rm \mathcal{T}}_{n}(1;x)-1\right|\nonumber\\&\leq& \frac{\epsilon}{p(t)}\sum\limits_{n=1}^{\infty}p_{n}t^{n-1}\overline{\rm \mathcal{T}}_{n}\left(1;x\right)+M_{\alpha}^{\pm}\left|\frac{1}{p(t)}\sum\limits_{n=1}^{\infty}p_{n}t^{n-1}\overline{\rm \mathcal{T}}_{n}(1;x)-1\right|\nonumber\\&+&\frac{2 M_{\alpha}^{\pm}}{\delta^2 p(t)}\sum\limits_{n=1}^{\infty}p_{n}t^{n-1}\overline{\rm \mathcal{T}}_{n}\left((z-x)^2;x\right),\nonumber
\end{eqnarray}
which yields
\begin{eqnarray}
\left|\frac{1}{p(t)}\sum\limits_{n=1}^{\infty}\overline{\rm \mathcal{T}}_{n}(f_{\alpha}^{\pm};x)p_{n}t^{n-1}-f_{\alpha}^{\pm}(x)\right|
&\leq& \epsilon+(M_{\alpha}^{\pm}+\epsilon+\frac{2d^2M_{\alpha}^{\pm}}{\delta^2})\left|\frac{1}{p(t)}\sum\limits_{n=1}^{\infty}p_{n}t^{n-1}\overline{\rm \mathcal{T}}_{n}(e_0;x)-e_0\right|\nonumber\\&+&\frac{ 2M_{\alpha}^{\pm}}{\delta^2}\bigg\{\left|\frac{1}{p(t)}\sum\limits_{n=1}^{\infty}p_{n}t^{n-1}\overline{\rm \mathcal{T}}_{n}\left(e_2;x,y\right)-e_2(x)\right|\nonumber\\&+&\frac{4dM_{\alpha}^{\pm}}{\delta^2}\left|\frac{1}{p(t)}\sum\limits_{n=1}^{\infty}p_{n}t^{n-1}\overline{\rm \mathcal{T}}_{n}\left(e_1;x\right)-e_1(x)\right|\nonumber\\&\leq& \epsilon+K_{\alpha}^{\pm}\sum\limits_{i=0}^2\left|\frac{1}{p(t)}\sum\limits_{n=1}^{\infty}\overline{\rm \mathcal{T}}_{n}(e_i;x)p_{n}t^{n-1}-e_i(x)\right|,
\end{eqnarray}
where $K_{\alpha}^{\pm}=\{\frac{2M_{\alpha}^{\pm}}{\delta^2},\frac{4dM_{\alpha}^{\pm}}{\delta^2},\epsilon+M_{\alpha}^{\pm}+\frac{2d^2M_{\alpha}^{\pm}}{\delta^2}\}$ and $d=\max\{|a|,|b|\}.$ Hence by using definition of the metric D and in view of the property (\ref{vetia1}) we get
\begin{eqnarray}
D\left(\frac{1}{p(t)}\sum\limits_{n=1}^{\infty}p_{n}t^{n-1} \mathcal{T}_{n}(f),f(x)\right)&\leq&\epsilon+K\bigg\{\sum\limits_{i=0}^2\left|\frac{1}{p(t)}\sum\limits_{n=1}^{\infty}\overline{\rm \mathcal{T}}_{n}(e_i;x)p_{n}t^{n-1}-e_i(x)\right|\bigg\},\nonumber
\end{eqnarray}
where $K=K(\epsilon)=\sup\limits_{\alpha\in[0,1]}\max\left\{K_{\alpha}^{+},K_{\alpha}^{-}\right\}.$ Then taking supremum over $x\in I$, we have
\begin{eqnarray}
D^*\left(\frac{1}{p(t)}\sum\limits_{n=1}^{\infty}p_{n}t^{n-1} \mathcal{T}_{n}(f),f\right)&\leq&\epsilon+K\bigg\{\sum\limits_{i=0}^2||\frac{1}{p(t)}\sum\limits_{n=1}^{\infty}\overline{\rm \mathcal{T}}_{n}(e_i)p_{n}t^{n-1}-e_i||\bigg\}.\nonumber
\end{eqnarray}
Finally, by taking limit as $t\to 1^-$ on both sides and using hypothesis (\ref{eqthm0}), we complete the proof.
\end{proof}

Our next result is devoted to discuss fuzzy approximation by the operators $\mathcal{T}_{n}$ for any function $f\in C^{(\mathcal{F})}(J) $ by means of the fuzzy modulus of continuity based on the P-summability. From \cite{Anas1}, the fuzzy analogue of the first order modulus of continuity is defined as: 
\begin{eqnarray*}
\omega_1^{\mathcal{F}}(f;\delta)=\sup\limits_{z,x\in J;|z-x|\leq\delta}D(f(z),f(x)), \quad \text{for any}~ \delta >0, f\in C^{(\mathcal{F})}(J).
\end{eqnarray*}


\begin{lemma}\cite{Anastassiou}\label{LMM1}
Let $f\in C^{(\mathcal{F})}(J)$. Then for any $\delta>0,$
\begin{eqnarray*}
\omega_1^{\mathcal{F}}(f;\delta)=\sup\limits_{\alpha\in[0,1]}\max\{\omega_1(f_{\alpha}^{-};\delta),\omega_1(f_{\alpha}^{+};\delta)\}.
\end{eqnarray*}
\end{lemma}
  \begin{theorem}\label{thm1}
Consider a sequence of fuzzy positive linear operators $\mathcal{T}_{n}:C^{(\mathcal{F})}(J)\rightarrow C^{(\mathcal{F})}(J).$ Assume that there exists a corresponding sequence $\{\overline{\rm \mathcal{T}}_{n}\}_{n\in \mathbb{N}}$ from $C(J)$ into itself with the properties (\ref{vetia1}) and (\ref{fz1}). Further suppose that the following conditions hold:
\begin{enumerate}
\item $\lim\limits_{t\to 1^-}||\frac{1}{p(t)}\sum\limits_{n=1}^{\infty}\overline{\rm \mathcal{T}}_{n}(e_0) p_{n}t^{n-1}-e_0||=0,$
\item $\lim\limits_{t\to 1^-}\omega_{1}^{\mathcal{F}}(f;\gamma(t))=0, where\,\, \gamma(t)=\sqrt{||\frac{1}{p(t)}\sum\limits_{n=1}^{\infty}\overline{\rm \mathcal{T}}_{n}(\phi)p_{n}t^{n-1}||}\,\,with \,\,\phi(z)=(z-x)^2,\,\,for\,\, each\,\,x\in J. $
\end{enumerate}
Then, for all $f\in C^{(\mathcal{F})}(J)$, we have
$$\lim\limits_{t\to 1^-}D^*\left(\frac{1}{p(t)}\sum\limits_{n=1}^{\infty}\mathcal{T}_{n}(f)p_{n}t^{n-1},f\right)=0.$$
\end{theorem}
\begin{proof}
Let $f\in C^{(\mathcal{F})}(J)$ and $x\in J$ be arbitrary but fixed. Since $f_{\alpha}^{\pm}\in C(J)$, we can write, for every $\epsilon>0,$ that there exists a number $\delta>0$ such that  $|f_{\alpha}^{\pm}(z)-f_{\alpha}^{\pm}(x)|< \epsilon$ holds for every $z\in J,$ satisfying $|z-x|<\delta.$ Then for all $z\in J,$  we immediately get that
\begin{eqnarray}
|f_{\alpha}^{\pm}(z)-f_{\alpha}^{\pm}(x)|\leq\left(1+\frac{(z-x)^2}{\delta^2}\right)\omega_1(f_{\alpha}^{\pm};\delta)\nonumber
\end{eqnarray}
and hence we obtain
\begin{eqnarray}
\left|\frac{1}{p(t)}\sum\limits_{n=1}^{\infty}p_{n}\overline{\rm \mathcal{T}}_{n}(f_{\alpha}^{\pm}(z);x)t^{n-1}-f_{\alpha}^{\pm}(x)\right|&\leq& \frac{1}{p(t)}\sum\limits_{n=1}^{\infty}p_{n}t^{n-1}\overline{\rm \mathcal{T}}_{n}\left(\left|f_{\alpha}^{\pm}(z)-f_{\alpha}^{\pm}(x)\right|;x\right)\nonumber\\&+&M_{\alpha}^{\pm}\left|\frac{1}{p(t)}\sum\limits_{n=1}^{\infty}p_{n}t^{n-1}\overline{\rm \mathcal{T}}_{n}(e_0;x)-e_{0}(x)\right|\nonumber\\&\leq&\frac{\omega_1(f_{\alpha}^{\pm};\delta)}{p(t)}\sum\limits_{n=1}^{\infty}p_{n}t^{n-1}\overline{\rm \mathcal{T}}_{n}\left(1+\frac{(z-x)^2}{\delta^2};x\right)\nonumber\\&+&M_{\alpha}^{\pm}\left|\frac{1}{p(t)}\sum\limits_{n=1}^{\infty}p_{n}t^{n-1}\overline{\rm \mathcal{T}}_{n}(e_0;x)-e_{0}(x)\right|\nonumber\\&\leq&\omega_1(f_{\alpha}^{\pm};\delta)\left|\frac{1}{p(t)}\sum\limits_{n=1}^{\infty}p_{n}t^{n-1}\overline{\rm \mathcal{T}}_{n}\left(e_0;x\right)-e_0\right|+\omega_1(f_{\alpha}^{\pm};\delta)\nonumber\\&+&M_{\alpha}^{\pm}\left|\frac{1}{p(t)}\sum\limits_{n=1}^{\infty}p_{n}t^{n-1}\overline{\rm \mathcal{T}}_{n}(e_0;x)-e_{0}(x)\right|\nonumber\\&+&\frac{\omega_1(f_{\alpha}^{\pm};\delta)}{\delta^2 p(t)}\sum\limits_{n=1}^{\infty}p_{n}t^{n-1}\overline{\rm \mathcal{T}}_{n}\left(\phi(z);x\right),\nonumber
\end{eqnarray}
where $M_{\alpha}^{\pm}=||f_{\alpha}^{\pm}||.$ Then considering property (\ref{vetia1}), Lemma (\ref{LMM1}) and definition of the metric $D$ we obtain
\begin{eqnarray}
D\left(\frac{1}{p(t)}\sum\limits_{n=1}^{\infty}p_{n}{\rm \mathcal{T}}_{n}(f;x)t^{n-1},f(x)\right)&\leq&\left|\frac{1}{p(t)}\sum\limits_{n=1}^{\infty}p_{n}t^{n-1}\overline{\rm \mathcal{T}}_{n}\left(e_0;x\right)-e_0(x)\right|\omega_1^{\mathcal{F}}(f;\delta)\nonumber\\&+&\omega_1^{\mathcal{F}}(f;\delta)+M_{\alpha}^{\pm}\left|\frac{1}{p(t)}\sum\limits_{n=1}^{\infty}p_{n}\overline{\rm \mathcal{T}}_{n}(e_0;x)t^{n-1}-e_{0}(x)\right|\nonumber\\&&+\frac{\omega_1(f_{\alpha}^{\pm};\delta)}{\delta^2 p(t)}\sum\limits_{n=1}^{\infty}p_{n}t^{n-1}\overline{\rm \mathcal{T}}_{n}\left(\phi(z);x\right),\nonumber
\end{eqnarray}
where $M:=\sup\limits_{\alpha\in [0,1]}\max\{M_{\alpha}^{+},M_{\alpha}^{-}\}.$ Taking supremum over $x\in I$ and putting $\delta=\gamma(t),$ we conclude that
\begin{eqnarray}
D^*\left(\frac{1}{p(t)}\sum\limits_{n=1}^{\infty}p_{n}{\rm \mathcal{T}}_{n}(f)t^{n-1},f\right)&\leq&||\frac{1}{p(t)}\sum\limits_{n=1}^{\infty}p_{n}t^{n-1}\overline{\rm \mathcal{T}}_{n}\left(e_0\right)-e_0||\omega_1^{\mathcal{F}}(f;\gamma(t))\nonumber\\
&+&2\omega_1^{\mathcal{F}}(f;\gamma(t))\nonumber\\
\end{eqnarray}
\begin{eqnarray}
&+&M||\frac{1}{p(t)}\sum\limits_{n=1}^{\infty}p_{n}\overline{\rm \mathcal{T}}_{n}(e_0)t^{n-1}-e_{0}||\nonumber\\&\leq& K\bigg\{||\frac{1}{p(t)}\sum\limits_{n=1}^{\infty}p_{n}\overline{\rm \mathcal{T}}_{n}(e_0)t^{n-1}-e_{0}||\omega_1^{\mathcal{F}}(f;\gamma(t))\nonumber\\&+&\omega_1^{\mathcal{F}}(f;\gamma(t))+||\frac{1}{p(t)}\sum\limits_{n=1}^{\infty}p_{n}\overline{\rm \mathcal{T}}_{n}(e_0)t^{n-1}-e_{0}||\bigg\}\nonumber
\end{eqnarray}
where $K=\max\{M,2\}.$ Finally, taking the limit as $t\to 1^-$, and considering the hypotheses (i) and (ii) of the theorem, the proof of the theorem is completed.
\end{proof}
\section{Illustrative Example}
\noindent In the following example, we elaborate that our theorem \ref{thm.asym.voronovaskaya1} is a non-trivial generalization of the classical fuzzy Korovkin type result given in \cite{Anas1}. 

\noindent\textbf{Example 1.} Let us consider the  operators of Fuzzy-Bernstein defined as:
$$\mathfrak{B}_{i}^{\mathcal{F}}(f;x)=\bigoplus_{j=0}^i{i\choose j}x^j(1-x)^{i-j}\bigodot f\left(\frac{j}{i}\right)$$
where $f \in C^{\mathcal{F}} [0, 1]$, $x\in [0,1]$ and $i\in\mathbb{N}$.
 Using these operators, for $f \in C^{\mathcal{F}} [0, 1],$ let us consider the fuzzy positive linear operators on $C^{\mathcal{F}} [0, 1]$ defined as follows:
\begin{eqnarray}\label{sheq1}
\overline{\rm{B}_{i}^{\mathcal{F}}}(f;x)=(1+x_i) \bigodot \mathfrak{B}_{i}^{\mathcal{F}}(f;x)
\end{eqnarray}
where $\big<x_i\big>=
\begin{array}{cc}
  \bigg\{\begin{array}{cc}
  1 ,& \text{if $i=m^3,$ $m\in \mathbb{N}$ } ,\\
  0,& \text{otherwise} \\
    \end{array}
\end{array}
$. Then, we have
$$\bigg\{\overline{\rm{B}_{i}^{\mathcal{F}}}(f;x)\bigg\}_{\pm}^{(r)}=\overline{\rm{B}_{i}}(f_{\pm}^{(r)};x)=(1+x_i)\sum_{j=0}^i{i\choose j}x^j(1-x)^{i-j} f_{\pm}^{(r)}\left(\frac{j}{i}\right)$$
It is clear that the sequence $({x_i})$ diverges in the classical sense.
Further, we observe that
$$\overline{\rm\mathfrak{B}_{i}}(e_0;x)=(1+x_i);$$
$$\overline{\rm\mathfrak{B}_{i}}(e_1;x)=(1+x_i)x;$$
$$\overline{\rm\mathfrak{B}_{i}}(e_2;x)=(1+x_i)\left(x^2+\frac{x(1-x)}{i}\right).$$
Now if we take $p_i=1$ for all $i\in\mathbb{N},$ then we obtain
$p(x)=\sum_{n=1}^{\infty}p_n x^{n-1}=\frac{1}{1-x},\,|x|<1$ which implies that $R=1.$ Further, we obtain
\begin{eqnarray*}
\lim_{x\to 1-}\frac{1}{p(x)}\sum_{i=1}^{\infty}p_i x^{i-1}x_i&=&\lim_{x\to 1-}\frac{(1-x)}{x}\sum_{m=1}^{\infty}x^{m^3}.
\end{eqnarray*}
Applying Cauchy root test, the series $\sum_{m=1}^{\infty}x^{m^3}$, is absolutely convergent in $|x|<1$, therefore it follows that
\begin{eqnarray}\label{remeq2}
\lim_{x\to 1^-}\frac{1}{p(x)}\sum_{i=1}^{\infty}p_ix^{i-1}x_i&=&\lim_{x\to 1-}\frac{(1-x)}{x}\sum_{m=1}^{\infty}x^{m^3}=0.
\end{eqnarray}
Hence,
 using $\overline{\rm\mathfrak{B}_{i}}(e_0;x)=(1+x_i),$, we conclude that
$$\lim\limits_{x\to 1^-}\frac{1}{p(x)}\sum\limits_{i=1}^{\infty}p_ix^{i-1}||\overline{\rm\mathfrak{B}_{i}}(e_0)-e_0||=0.$$
Moreover, from $\overline{\rm\mathfrak{B}_{i}}(e_1;x)=(1+x_i)x,$ we have
$$||\overline{\rm\mathfrak{B}_{i}}(e_1)-e_1||\leq  x_i,\;\forall\;i\in \mathbb{N}.$$
Hence, in view of equation (\ref{remeq2}) we get
$$\lim\limits_{x\to1^-}\frac{1}{p(x)}\sum\limits_{i=1}^{\infty}p_i x^{i-1}||\overline{\rm\mathfrak{B}_{i}}(e_1)-e_1||=0.$$
Finally, we have
\begin{eqnarray}
||\overline{\rm\mathfrak{B}_{k}}(e_2)-e_2||&\leq&\frac{1}{4i}+x_i\left(1+\frac{1}{4i}\right)=\tau_1+\tau_2.\nonumber
\end{eqnarray}
Since $\tau_1\to 0,$ as $i\to \infty$, it follows that it converges to $0$ in terms of the power series method. Further, since $\left(1+\frac{1}{4i}\right)$ is a convergence sequence as $i\to \infty,$ it is bounded and therefore in view of (\ref{remeq2}), $\tau_2$ converges to $0$ in terms of the power series method.
Consequently, we obtain
$$\lim\limits_{z\to1^-}\frac{1}{p(x)}\sum\limits_{i=1}^{\infty}p_i x^{i-1}||\overline{\rm\mathfrak{B}_{i}}(e_2)-e_2||=0.$$
So, by Theorem \ref{thm.asym.voronovaskaya1}, we obtain, for all $f\in C^{(\mathcal{F})}[0,1]$ that
\begin{eqnarray}\label{eqthm1}
\lim\limits_{t\to 1^-}D^{*}\left(\frac{1}{p(t)}\sum\limits_{n=1}^{\infty}p_{n}t^{n-1} \overline{\rm{B}_{i}^{\mathcal{F}}}(f),f\right)=0.
\end{eqnarray}
However, it can be said that the fuzzy Korovkin Theorem \ref{thm.asym.voronovaskaya} does not work for the operators defined by (\ref{sheq1}), since $(x_i)$ is not convergent to 0 (in the usual sense).

\section{Conclusions and Comments}
\noindent In the paper, we have made efforts to establish a power series summability analogue of a basic fuzzy Korovkin type approximation theorem. Further with some suitable assumptions, we have proved an approximation theorem based on the first order fuzzy modulus of continuity. In addition to this, we found an example which shows the significance of our study. Interestingly, one can extend our study in the modular spaces (or in the multivariate settings too). It would be also nice to see the new developments in our results through the other known summability techniques.             

\begin{center}
\textbf{Acknowledgements}
\end{center}
The third author is heartily indebted to \textbf{Prof. Ram Narayan Mohapatra, Department of Mathematics, University of Central Florida} for all the invaluable suggestions, motivation and unconditional support to complete this manuscript.

\newpage

\end{document}